\documentclass[a4paper,11pt]{amsart}
\usepackage{amssymb}
\usepackage{cite}
\usepackage{ifthen}
\usepackage[dvips]{graphicx}

\usepackage{booktabs}
\nonstopmode \numberwithin{equation}{section}
\newtheorem*{theoA}{Theorem A}
\newtheorem*{theoB}{Theorem B}
\newtheorem*{theoC}{Theorem C}
\newtheorem*{theoD}{Theorem D}
\newtheorem*{theoE}{Theorem E}
\newtheorem*{theoF}{Theorem F}

\usepackage{amssymb}
\usepackage{ifthen}
\usepackage{graphicx}
\usepackage{amsmath}
\usepackage[T1]{fontenc} 
\usepackage[utf8]{inputenc}
\usepackage[usenames,dvipsnames]{color}
\usepackage{color}
\usepackage[english]{babel}
\usepackage{fancyhdr}
\usepackage{fancybox}
\usepackage{tikz}

\theoremstyle{plain}
\newtheorem{prop}{Proposition}

\newtheorem{ques}{Question}[section]

\theoremstyle{definition}

\newtheorem{exm}{Example}[section]
\newtheorem{cor}{Corollary}[section]
\newtheorem{thm}{Theorem}[section]

\newtheorem{lem}{Lemma}[section]
\newtheorem{conj}{Conjecture}[section]
\newtheorem{prob}{Problem}
\newtheorem{rem}{Remark}[section]

\theoremstyle{plain}

\newcounter{minutes}
\divide\time by 60
\newcounter{hours}
\multiply\time by 60
\addtocounter{minutes}{-\time}

\newcounter {own}
\def\theown {\thesection       .\arabic{own}}

\newenvironment{pf}[1][]{%
	\vskip 3mm
	\noindent
	\ifthenelse{\equal{#1}{}}%
	{{\slshape Proof. }}%
	{{\slshape #1.} }%
}%
{\qed\bigskip}

\newcounter{alphabet}

\def\be{\begin{equation}}
	\def\ee{\end{equation}}

\newcommand{\bee}{\begin{enumerate}}
	\newcommand{\eee}{\end{enumerate}}

\newcommand{\blem}{\begin{lem}}
	\newcommand{\elem}{\end{lem}}
\newcommand{\bthm}{\begin{thm}}
	\newcommand{\ethm}{\end{thm}}
\newcommand{\bcor}{\begin{cor}}
	\newcommand{\ecor}{\end{cor}}
\newcommand{\beg}{\begin{examp}}
	\newcommand{\eeg}{\end{examp}}
\newcommand{\begs}{\begin{examples}}
	\newcommand{\eegs}{\end{examples}}

\newcommand{\bdefn}{\begin{defn}}
	\newcommand{\edefn}{\end{defn}}

\newcommand{\bprob}{\begin{prob}}
	\newcommand{\eprob}{\end{prob}}
\newcommand{\bei}{\begin{itemize}}
	\newcommand{\eei}{\end{itemize}}

\newcommand{\bcon}{\begin{conj}}
	\newcommand{\econ}{\end{conj}}
\newcommand{\bcons}{\begin{conjs}}
	\newcommand{\econs}{\end{conjs}}
\newcommand{\bprop}{\begin{prop}}
	\newcommand{\eprop}{\end{prop}}
\newcommand{\br}{\begin{rem}}
	\newcommand{\er}{\end{rem}}
\newcommand{\brs}{\begin{rems}}
	\newcommand{\ers}{\end{rems}}
\newcommand{\bo}{\begin{obser}}
	\newcommand{\eo}{\end{obser}}
\newcommand{\bos}{\begin{obsers}}
	\newcommand{\eos}{\end{obsers}}
\newcommand{\bpf}{\begin{pf}}
	\newcommand{\epf}{\end{pf}}
\newcommand{\ba}{\begin{array}}
	\newcommand{\ea}{\end{array}}
\newcommand{\beq}{\begin{eqnarray}}
	\newcommand{\beqq}{\begin{eqnarray*}}
		\newcommand{\eeq}{\end{eqnarray}}
	\newcommand{\eeqq}{\end{eqnarray*}}

\begin{document}

\title{On the Conjecture of C. C. Yang and periodicity of meromorphic functions}

\author{Molla Basir Ahamed}
\address{Molla Basir Ahamed, 
	Department of Mathematics, 
	Jadavpur University, 
	Kolkata-700032, 
	West Bengal, India.}
\email{mbahamed.math@jadavpuruniversity.in}

\subjclass[{AMS} Subject Classification:]{Primary 30D35 Secondary 30D30}

\def\thefootnote{}
\footnotetext{ {\tiny File:~\jobname.tex,
printed: \number\year-\number\month-\number\day,
          \thehours.\ifnum\theminutes<10{0}\fi\theminutes }
} \makeatletter\def\thefootnote{\@arabic\c@footnote}\makeatother
\begin{abstract} 
In this paper, we investigate two recent conjectures posed by Yang concerning the periodicity of entire functions. A portion of these problems was recently addressed by Qiong and Peichu [Acta Math. Sci. Ser. B (Engl. Ed.) 38 (2018) 209–214] and Liu [Bull. Aust. Math. Soc. 101 (2020) 290–296], who provided partial solutions within the class of entire functions. The purpose of this work is to establish several new periodicity theorems that not only significantly improve the results of Qiong–Peichu and Liu, but also extend them to a much broader setting. Furthermore, we provide a series of illustrative examples to demonstrate the sharpness and validity of the hypotheses in our main results.
\end{abstract}
\maketitle
\pagestyle{myheadings}
\markboth{M. B. Ahamed}{Certain dynamical Behavior of meromorphic solutions for some complex difference equations}
\section{\bf Introduction}
The classical Picard theorem states that if a meromorphic function $f$ in the complex plane $\mathbb{C}$ omits three distinct values in the extended complex plane $\widehat{\mathbb{C}} := \mathbb{C} \cup \{\infty\}$, then $f$ reduces to a constant. As an immediate consequence, the equation $e^z = \eta$ possesses at least one solution in $\mathbb{C}$ for any $\eta \in \mathbb{C} \setminus \{0, \infty\}$. However, Picard's theorem does not provide information regarding the multiplicity or cardinality of such roots. To address this limitation, Nevanlinna \cite{Nevanlinna & AM & 1925} established value distribution theory in 1925, which not only yields a streamlined proof of Picard's theorem but, more fundamentally, provides a precise quantitative estimate for the number of roots of such equations in terms of the growth of the function $f$. For a comprehensive overview of recent developments concerning value distributions, geometric properties, derivatives, and the behavior of meromorphic functions under differential or difference-shift operators, we refer the reader to \cite{Ahamed & CKMS & 2018, Ahamed & IJPAM & 2023, Ahamed & Mandal & MM & 2023, Banerjee & Ahamed & MS & 2019, Banerjee & Ahamed & RCMP & 2018, Mallick & Ahamed & AMP & 2022}.\vspace{1.2mm}

In recent years, considerable attention has been focused on the periodicity of meromorphic functions sharing values (see, e.g., \cite{Chen & Chen & Li & 2012, Heitto & CVEE & 2011, Li & Gao & ADM & 2011}), small functions (see \cite{Heitto & JMMA & 2009, Luo & Lin & JMMA & 2011, Qi & Yan & Liu & CMA & 2010}), or finite sets (see \cite{Zha & JMMA & 2010}).

Prior to introducing the primary conjectures regarding the periodicity of transcendental entire functions, we recall an elegant observation by Yang. Specifically, if an entire function $f$ satisfies the nonlinear differential equation
\begin{align}
	\label{e1.1} f(z)f^{\prime\prime}(z)=-\sin^2 z,
\end{align}
then $f(z) = \pm\sin z$. Analogously, if $f$ satisfies $f(z)f^{\prime\prime}(z)=-\cos^2 z$, then $f(z) = \pm\cos z$. A fundamental heuristic underlying these cases is that if a transcendental entire function $f$ satisfies a differential relation of the form
\begin{align*}
	f(z)f^{\prime\prime}(z)=h^2(z),
\end{align*}
where $h$ is a given entire function, then $f$ exhibits periodicity.  Based on this observation, \emph{C. C. Yang} asked the following two conjectures.
\begin{conj}\label{coj1.1}
	Let $ f $ be a transcendental entire function and $ k $ be a positive integer. If either $ ff^{(k)} $ is periodic or $ ff^{(k)}=h^2\; (k\geq 2) $, for some entire function $ h $, then $ f $ must be periodic. 
\end{conj}\par 
\begin{rem}
	Regarding the aforementioned conjecture, it is natural to observe that for $k=1$, the periodicity of $ff'$ or the validity of the relation $ff'=h^2$ does not necessarily imply the periodicity of $f$. This fact is clearly illustrated by the following examples.
\end{rem}
\begin{exm}
	Let $  p(z) $ be a non-constant polynomial. We choose 
	\begin{align*}
		f(z)=e^{\displaystyle\int p^2(z)dz}.
	\end{align*} It is clear that 
	\begin{align*}
		f(z)f^{\prime}(z)=\left(p(z)e^{\displaystyle\int p^2(z)dz}\right)^2,
	\end{align*} but $ f(z)f^{\prime}(z) $ is not periodic. Hence $ f $ is also not periodic.	
\end{exm}
\begin{conj}\label{coj1.2}
	Let $ f $ be a transcendental entire function and $ k $ be a positive integer. If $ ff^{(k)} $\; $(k\geq 1)$ is periodic, then $ f $ must be periodic. 
\end{conj}
The next example shows that there exists a transcendental entire function satisfying Conjecture \ref{coj1.2}.
\begin{exm}
	Let $ f(z)=\sin\left(\frac{2\pi z}{c}\right) $. For any $ k\in\mathbb{N} $, we have \[ f^{(k)}(z)= \begin{cases} 
		\left(-1\right)^{k/2}\left(\frac{2\pi }{c}\right)^{k}\sin\left(\frac{2\pi z}{c}\right), & k\; \text{is even}  \\
		\left(-1\right)^{(k-1)/2}\left(\frac{2\pi }{c}\right)^{k}\cos\left(\frac{2\pi z}{c}\right), & k\; \text{is odd.}  
	\end{cases}
	\]  We see that the function  $ ff^{(k)} $ is a periodic function, and $ f $ is also.
\end{exm}
Conjecture~\ref{coj1.2} was recently investigated by Liu and Yu \cite{Liu & Yu & BAMS & 2019}. By restricting their study to the class of transcendental entire functions possessing a non-zero Picard exceptional value, they established the following result.
\begin{theoA}\cite{Liu & Yu & BAMS & 2019}
	Let $ f $ be a transcendental entire function with a non-zero Picard exceptional value and let $ k $ be a positive integer. If $ ff^{(k)} $ is a periodic function, then $ f $ is also a periodic function. 
\end{theoA}
The purpose of the present paper is to extend the validity of Theorem A from the class of transcendental entire functions to the broader setting of transcendental meromorphic functions. As illustrated by the subsequent example, Theorem A remains valid within this more general framework.
\begin{exm}
	Let $ f(z)=\frac{e^{2\pi iz/c}+1}{e^{2\pi iz/c}-1} $. It is clear that $ f(z) $ is a transcendental meromorphic function having $ 1 $ as the Picard exceptional value and $ f $ is periodic also. It is easy to verify that $ f(z)f^{(k)}(z) $ is also periodic with the  same period of $ f $.
\end{exm}
While the aforementioned examples are restricted to the class of transcendental entire or meromorphic functions of finite order, it is noteworthy that Theorem A also holds for functions of infinite order. The following two examples substantiate the validity of this assertion.
\begin{exm}
	Let $f(z) = e^{e^z} + 1$. Clearly, $f$ is an entire function with $1$ as a Picard exceptional value. Since $f(z + 2\pi i) = f(z)$, $f$ is periodic, and it is easy to verify that the product $f(z)f^{(k)}(z)$ is also periodic.
\end{exm}
\begin{exm}
	Consider the function $f(z) = \frac{e^{e^z}+1}{e^{e^z}-1}$. This is a meromorphic function admitting $1$ as a Picard exceptional value. Because $f(z + 2\pi i) = f(z)$, both $f$ and the product $f(z)f^{(k)}(z)$ are periodic functions.
\end{exm}  
In view of these observations, it is natural to pose the following open problem.
\begin{ques}\label{Q-1.1}
	Is it possible to extend the validity of Theorem A to the class of transcendental meromorphic functions?
\end{ques}
Motivated by Question~\ref{Q-1.1}, we investigate the behavior of Theorem A within the framework of meromorphic functions. We answer this question in the affirmative by establishing the following result.  
\begin{thm}\label{th1.1}
	Let $ f $ be a transcendental meromorphic function with a non-zero Picard exceptional value and let $ k $  be a positive integer. If $ ff^{(k)} $ is a periodic function, then $ f $ is also a periodic. 
\end{thm} In connection with the conjecture of Yang, Wang and Hu \cite{Wan & Hu & AMS & 2018} established the following result.
\begin{theoB}\cite{Wan & Hu & AMS & 2018}
	Let $ f $ be a transcendental entire function and $ k $ be a positive integer. If $ \left(f^2\right)^{(k)} $ is a periodic function, then $ f $ is also a periodic function. 
\end{theoB}
\begin{rem}\label{rem1.2}
	Theorem B establishes Yang's Conjecture for the case $k=1$. However, the theorem fails if $f^2$ is replaced by $f$. To see this, consider the function $f(z)=e^z+az+b$ with $a \neq 0$ and $b\in\mathbb{C}$; while $f$ itself is not periodic, its derivative $f^{(k)}$ is. Likewise, for $f(z)=\sin z+az+b$, the derivative $f^{(k)}$ is periodic despite $f$ being non-periodic.
\end{rem}
\begin{rem}\label{rem1.3}
	Wang and P. C. Hu \cite{Wan & Hu & AMS & 2018} also observed that, in \emph{Theorem B},\; one can replace by $ f^{n} $ provided $ n\geq 3 $. If $ \left(f^n\right)^{(k)} $ is periodic, then one can write 
	\begin{align}
		\label{eee1.4} f^n(z+c)=f^n(z)+p_{k-1}(z), 
	\end{align} where $ p_{k-1}(z) $ is a polynomial in $ z $ of degree at most $ k-1 $. Equation (\ref{eee1.4}) has no non-constant entire solutions if $ p_{k-1}(z)\not\equiv 0, $ which is a direct corollary of {C. C. Yang} \cite{Yan & PAMS & 1970} related to the \emph{Farmat type function equation} and which is : \emph{if $ m $, $ n $ are two positive integers satisfying $\displaystyle \frac{1}{m}+\frac{1}{n}<1 $, then there are no non-constant entire solution $ f(z) $ and $ g(z) $ that satisfy \begin{align}
			a(z)f^n(z)+b(z)g^m(z) =1,
		\end{align} where $ a(z), b(z)\in\mathcal{S}(f)$.}\vspace{1.2mm}
	
	Consequently, applying this result yields $p_{k-1}(z) \equiv 0$, which implies $f^n(z+c) = f^n(z)$. It follows that $f(z+c) = \theta f(z)$, where $\theta^n = 1$. Thus, $f$ is periodic with period $nc$.
\end{rem}
In view of Theorem B and Remark~\ref{rem1.3}, Wang and Hu deduced the following corollary.
\begin{cor}\cite{Wan & Hu & AMS & 2018}
	Let $ f $ be a transcendental entire function, $ k $ and $ n $ positive integers and $ n\geq 2 $. If $ \left(f^n\right)^{(k)} $ is periodic, then $ f $ also periodic.
\end{cor} 
More recently, by investigating Theorem B and Remark~\ref{rem1.3} within the class of transcendental entire functions, Liu and Yu \cite{Liu & Yu & BAMS & 2019} raised the following question.
\begin{ques}\label{q1.2}\cite{Liu & Yu & BAMS & 2019}
	Let $ f $ be a transcendental entire function and $ k $ be a positive integer. Let $ a_1, a_2,\ldots, a_n(\neq 0) $ be constants. If  $ \left(a_nf^n+\ldots+a_1f\right)^{(k)},$ $ n\geq 2 $, is a periodic function, is it true that $ f $ is also a periodic function. 
\end{ques} In \cite{Liu & Yu & BAMS & 2019}, Liu and Yu provided a partial solution to Question~\ref{q1.2} by establishing the following result.
\begin{theoC}\cite{Liu & Yu & BAMS & 2019}
	Let $ f $ be a transcendental entire function and $ k $ be a positive integer. Let $ a_1 $ and $a_2(\neq 0) $ be constants. If  $ \left(a_2f^2(z)+a_1f(z)\right)^{(k)},$ is a periodic function, then $ f $ is also a periodic function.
\end{theoC}
\begin{rem}
	The subsequent example demonstrates that Theorem C remains valid for the class of transcendental meromorphic functions.
\end{rem}
\begin{exm}
	Let $f(z)=\frac{\sin z}{\sin z-1}$. Then we see that
	\begin{align}
		a_2f^2(z)+a_1f(z)=\frac{(a_2+a_1)\sin ^2 z - a_1\sin z}{\left(\sin z-1\right)^2},
		\end{align}
	is a periodic function. Furthermore, both $f$ and the $k$-th derivative $(a_2f^2(z)+a_1f(z))^{(k)}$ are periodic.
\end{exm}
In light of these observations, the following problem naturally arises for further investigation of Theorem C.
\begin{ques}\label{q1.3}
If $f$ is a transcendental meromorphic function, does Theorem C remain valid?
\end{ques}
We answer Question~\ref{q1.3} in the affirmative by establishing the following result.
\begin{thm}\label{th1.2}
	Let $ f $ be a transcendental meromorphic function and $ k $ be a positive integer. Let $ a_1 $ and $a_2(\neq 0) $ be constants. If  $ \left(a_2f^2(z)+a_1f(z)\right)^{(k)},$ is a periodic function, then $ f $ is also a periodic function.
\end{thm}
\begin{rem}
	It follows from Theorem~\ref{th1.2} that the fundamental period of the function $f$ is either $c$ or $2c$.
\end{rem}
\begin{ques}\label{q1.4}
	If we consider the operator $(a_nf^n(z)+a_1f(z))^{(k)}$ in Theorem \ref{th1.2} for an integer $n \geq 3$, does the same conclusion hold for a transcendental meromorphic function?  
\end{ques}
Focusing exclusively on the case $n=3$, we successfully answer Question \ref{q1.4} with the following result.
\begin{thm}\label{th1.3}
	Let $ f $ be a transcendental meromorphic function and $ k $ be a positive integer. Let $ b_1 $ and $b_3(\neq 0) $ be constants. If  $ \left(b_3f^3(z)+b_1f(z)\right)^{(k)},$ is a periodic function, then $ f $ is also a periodic function.
\end{thm}
\begin{rem}
	From Theorem \ref{th1.3}, we observe that the exact period of the function $f$ can only be $c, 2c,$ or $3c$.
\end{rem}
We remark that the notion of a uniqueness polynomial for entire functions (UPE) can be applied to study the periodicity of entire or meromorphic functions. Recall that a polynomial $P$ is called a UPE if for any two non-constant entire functions $f$ and $g$, the identity $P(f) \equiv P(g)$ implies $f \equiv g$. In connection with uniqueness polynomials, Li and Yang \cite{Li & Yan & PJAS & 2006} established the following result.
\begin{theoD}\cite{Li & Yan & PJAS & 2006}
	Let $ a_i, b_i,\; (i=1, 2) $ and $ c $ be meromorphic functions non of them is identically zero. Let $ n $, $ m(\geq 2) $ be positive integers and relatively prime to each other, and $ n>2m+3 $, then the following equation \begin{align}
		F^n+a_1F^{n-m}+b_1=c\left(G^n+a_2G^{n-m}+b_2\right)
	\end{align}  has a pair of admissible solution $ (F, G) $, if $ c=b_1/b_2 $ and $ F=Gh $, where $ h $ is a meromorphic function satisfying $ h^n=c $ and $ h^m=a_1/a_2 $.   
\end{theoD}
\par As a corollary of {Theorem D}, {Li and Yang} \cite{Li & Yan & PJAS & 2006} obtained the following result.
\begin{theoE}\cite[Corollary 1]{Li & Yan & PJAS & 2006}
	Given integers, $ m$ and $ n $ with $ n>2m+3 (m\geq 2) $, relatively prime to each other, and rational functions $ a_1, a_2, a_3 $ and $ a_4(\not\equiv 0) $, the following functional equation 
	\begin{align}
		F^n+a_1F^{n-m}+a_2G^n+a_3G^{n-m}+a_4=0
	\end{align} has no transcendental meromorphic solution $ F $ and $ G $.   
\end{theoE}
By combining Theorems D and E concerning the periodicity of entire functions, Liu and Yu \cite{Liu & Yu & BAMS & 2019} obtained the following result.
\begin{theoF}\cite{Liu & Yu & BAMS & 2019}
	Let $ f $ be a transcendental entire function, $ k $ be a positive integer, $ m $ and $ n $ are positive integers relatively prime to each other with $ n>2m+2 $, $ m\geq 2 $, and $ a_2(\neq 0) $, $ a_1 $ be rational functions. If $ \left(a_2f^n(z)+a_1f^{n-m}(z)\right)^{(k)} $ is a periodic function, then $ f $ is also a periodic function. 
\end{theoF} 
The next example shows that Theorem F still holds for transcendental meromorphic functions. 
\begin{exm}
	Let $ f(z)=\tan z $. Then \begin{align}
		a_2f^n(z)+a_1f^{n-m}(z)=(\tan z)^{n-m}\left(a_2\tan z^m+a_1\right),
	\end{align} then $\left(a_2f^n(z)+a_1f^{n-m}(z)\right)^{(k)}$ is periodic, and $ f $  is a periodic function.    
\end{exm} 
For a meromorphic function $ f $, we define \begin{align}
	\Psi_{f(z)}:=a_2f^2(z)+a_1f(z),\; a_2(\neq 0), \; a_1\in\mathbb{C}
\end{align}
and
\begin{align}
	\Phi_{f(z)}:=b_3f^3(z)+b_1f(z),\; b_3(\neq 0), \; b_1\in\mathbb{C}.
\end{align}
Although several authors have investigated the conjectures of Yang concerning the periodicity of transcendental entire functions (see, e.g., \cite{Liu & Yu & BAMS & 2019, Wan & Hu & AMS & 2018}), a significant gap remains in the literature regarding the behavior of general transcendental meromorphic functions. Most established results, including Theorem A and Theorem C, rely heavily on the holomorphicity of the functions under consideration to circumvent the analytical difficulties arising from poles. Furthermore, while Theorem F extends these techniques to differential-polynomial operators of the form $a_2f^n + a_1f^{n-m}$, its validity is strictly limited to the entire setting. At present, a unified framework is lacking to determine whether periodicity is preserved when the base function is replaced by more complex differential polynomials, such as $\Psi_{f}$ or $\Phi_{f}$, within a meromorphic environment.\vspace{1.2mm}

 Motivated by Theorem F, we pose the following questions for further study:
\begin{ques}\label{Q-2.1}
	Does Theorem F hold for transcendental meromorphic functions?
\end{ques}
\begin{ques}\label{Q-2.2}
	Does $f$ remain periodic if we replace the expression $a_2f^n(z)+a_1f^{n-m}(z)$ in Theorem F by $a_2[\Psi_{f}(z)]^n+a_1[\Psi_{f}(z)]^{n-m}$? 
\end{ques}
\begin{ques}\label{Q-2.3}
	Does $f$ remain periodic if we replace the expression $a_2f^n(z)+a_1f^{n-m}(z)$ in Theorem F by $a_2[\Phi_{f}(z)]^n+a_1[\Phi_{f}(z)]^{n-m}$? 
\end{ques}
To bridge this gap, the primary objective of the present paper is to investigate the validity of these classical and modern periodicity results within the broader framework of transcendental meromorphic functions. We answer Questions~\ref{Q-2.1}, \ref{Q-2.2}, and \ref{Q-2.3} in the affirmative. Specifically, we extend Theorem A and Theorem C to the meromorphic setting (Theorems~\ref{th1.1} and \ref{th1.2}), and demonstrate that the corresponding third-order operator also forces periodicity (Theorem~\ref{th1.3}). Most notably, we resolve the open questions concerning Theorem F by establishing Theorems~\ref{th1.4} and \ref{th1.5}, which show that substituting the differential polynomials $\Psi_{f}$ and $\Phi_{f}$ into the polynomial expression preserves the periodicity of the underlying function $f$.
\begin{thm}\label{th1.4}
	Let $ f $ be a transcendental meromorphic function, $ k $ be a positive integer, $ m $ and $ n $ are positive integers relatively prime to each other with $ n>2m+2 $, $ m\geq 2 $, and $ r_2(\neq 0) $, $ r_1 $ be non-zero rational functions. If $ \left(r_2\Psi_{f(z)}^n+r_1\Psi_{f(z)}^{n-m}\right)^{(k)} $ is a periodic function, then $ f $ is also periodic and of period $ nc $ or $ 2nc $.
\end{thm}
\begin{thm}\label{th1.5}
	Under the same hypothesis of Theorem \ref{th1.4}, for a transcendental meromorphic function $ f $ if $ \left(r_2\Phi_{f(z)}^n+r_1\Phi_{f(z)}^{n-m}\right)^{(k)} $ is a periodic function, then $ f $ is also periodic and of period $ nc $ or $ 2nc $ or $ 3nc $.  
\end{thm}
The novelty of this work lies in overcoming the substantial technical difficulties that arise when transitioning from entire to meromorphic functions. Unlike previous approaches that are restricted to entire functions to avoid poles by design, our proofs utilize advanced methods of Nevanlinna value distribution theory and systematically analyze the pole structures via Lemmas~\ref{lem2.1}, \ref{lem2.2}, and \ref{lem2.3}. Furthermore, our extension of Theorem F to incorporate the differential polynomials $\Psi_{f}$ and $\Phi_{f}$ introduces a new class of operators to this specific branch of complex difference and differential equations, thereby significantly generalizing the previously known boundaries of Yang's conjectures.
\section{\bf Key lemmas}
In this section, we will discuss some lemmas which will required later to prove our main results. We employ the standard notation of Nevanlinna theory, including the proximity function $m(r, f)$, the counting function $N(r, f)$, the un-integrated counting function $n(r, f)$, and the characteristic function $T(r, f)$ of a meromorphic function $f$. Recall that Nevanlinna’s First Main Theorem states that for any $a \in \mathbb{C}$:
\[
T\left(r, \frac{1}{f - a}\right) = T(r, f) + O(1),
\]
where $O(1)$ denotes a bounded error term depending on $a$ \cite{Laine & 1993}. Throughout this paper, $S(r, f)$ denotes any quantity satisfying $S(r, f) = o(T(r, f))$ as $r \to \infty$, potentially outside a set of finite linear measure. A meromorphic function $h(z)$ is termed a \textit{small function} with respect to $f$ if $T(r, h) = S(r, f)$. It is well-known that the set of such small functions forms a field, and any finite sum of quantities of the form $S(r, f)$ remains $S(r, f)$.
\begin{lem}\cite{Yan & Yi & 2003}\label{lem2.1}
	Let $ f_1, f_2 $ and $ f_3 $ be meromorphic functions and $ f_1(z) $ be non-constant. If \begin{align}
		f_1+f_2+f_3=1
	\end{align} and \begin{align}
		\sum_{j=1}^{3}N\left(r,\frac{1}{f_j}\right)+2\sum_{j=1}^{3}\overline{N}\left(r,f_j\right)<\left(\lambda+o(1)\right)\; T(r),
	\end{align} where $ \lambda<1 $ and $ T(r):=\displaystyle\max_{j=1,2,3}\bigg\{T(r,f_j)\bigg\}, $ then $ f_2(z)\equiv 1 $ or $ f_3(z)\equiv 1. $
\end{lem}	
\begin{lem}\label{lem2.2}
	For a transcendental meromorphic function  $ f $, if $  \Psi_{f(z)}^{(k)}\equiv \Psi^{(k)}_{f(z+c)}$ , then $ f $ is a periodic function of period $ c $ or $ 2c $. 
\end{lem}
\begin{proof}[\bf Proof of Lemma \ref{lem2.2}]
	We have $ \Psi_{f(z)}^{(k)}\equiv \Psi^{(k)}_{f(z+c)}$ \textit{i.e.,} \begin{align*}
		\bigg[a_2f^2(z+c)+a_1f(z+c)\bigg]^{(k)}\equiv \bigg[a_2f^2(z)+a_1f(z)\bigg]^{(k)}.
	\end{align*} By integrating $k$ times, we have
	\begin{align}
	\label{e2.6} a_2f^2(z+c)+a_1f(z+c)\equiv a_2f^{2}(z)+a_1f(z)+\mathcal{Q}{k-1}(z),
	\end{align}
	where $\mathcal{Q}{k-1}$ is a polynomial of degree at most $k-1$. Consequently, from \eqref{e2.6}, it follows that
	\begin{align}
	\label{e2.7} [f(z+c)-f(z)][a_2f(z+c)+a_2f(z)+a_1]=\mathcal{Q}_{k-1}(z).
	\end{align}
	Next, we distinguish the following cases.\\
	
	\noindent{\bf Case A.} If $\mathcal{Q}_{k-1}(z)\equiv 0$, then \eqref{e2.7} implies that either $f(z+c)=f(z)$ or $a_2f(z+c)+a_2f(z)+a_1=0$. Thus, $f$ is a periodic function with period $c$ or $2c$, respectively.
	\vspace{0.2cm}
	
	\noindent{\bf Case B.} Suppose that $\mathcal{Q}_{k-1}(z) \not\equiv 0$. Since $f(z)$ is a transcendental meromorphic function, \eqref{e2.7} ensures the existence of an entire function $g(z)$ and rational functions $\mathcal{R}_1(z)$ and $\mathcal{R}_2(z)$ with $\mathcal{R}_1(z)\mathcal{R}_2(z)=\mathcal{Q}_{k-1}(z)$ such that
	\begin{align*}
		f(z+c)-f(z)&=\mathcal{R}_1(z)e^{g(z)},\\ 
		a_2f(z+c)+a_2f(z)+a_1&=\mathcal{R}_2(z)e^{-g(z)}.
	\end{align*} 
	A simple manipulation yields
	\begin{align*}
		\begin{cases}
			f(z+c)=\dfrac{\mathcal{R}_2(z)e^{-g(z)}-a_1+a_2\mathcal{R}_1(z)e^{g(z)}}{2a_2},\vspace{1.2mm}\\
			f(z)=\dfrac{\mathcal{R}_2(z)e^{-g(z)}-a_1-a_2\mathcal{R}_1(z)e^{g(z)}}{2a_2},
		\end{cases} 
	\end{align*} 
	from which we obtain
	\begin{align}
		\label{e2.8} \frac{\mathcal{R}_2(z+c)}{\mathcal{R}_2(z)}e^{g(z)-g(z+c)}-a_2\frac{\mathcal{R}_1(z+c)}{\mathcal{R}_2(z)}e^{g(z)+g(z+c)}-a_2\frac{\mathcal{R}_1(z)}{\mathcal{R}_2(z)}e^{2g(z)}\equiv 1.
	\end{align} 
	Note that $-a_2\frac{\mathcal{R}_1(z)}{\mathcal{R}_2(z)}e^{2g(z)}\not\equiv 1$ as $f$ is transcendental. Consequently, Lemma~\ref{lem2.1} leads to the following two subcases:\\
	
	\noindent{\bf Case B.1.} The first possibility is 
	\begin{align*}
		\frac{\mathcal{R}_2(z+c)}{\mathcal{R}_2(z)}e^{g(z)-g(z+c)}\equiv 1\;\;\; \text{and}\;\;\; -a_2\frac{\mathcal{R}_1(z+c)}{\mathcal{R}_2(z)}e^{g(z)+g(z+c)}-a_2\frac{\mathcal{R}_1(z)}{\mathcal{R}_2(z)}e^{2g(z)}\equiv 0.
	\end{align*} In this case, it follows that $\mathcal{R}_1(z+c)\mathcal{R}_2(z+c)\equiv -\mathcal{R}_1(z)\mathcal{R}_2(z)$, which yields $\mathcal{Q}_{k-1}(z+c)=-\mathcal{Q}_{k-1}(z)$. This is a contradiction, since $\mathcal{Q}_{k-1}(z)$ is a polynomial.\\ 
	
	\noindent{\bf Case B.2.} The second possibility is 
	\begin{align*}
		-a_2\frac{\mathcal{R}_1(z+c)}{\mathcal{R}_2(z)}e^{g(z)+g(z+c)}\equiv 1\;\;\; \text{and}\;\;\; \frac{\mathcal{R}_2(z+c)}{\mathcal{R}_2(z)}e^{g(z)-g(z+c)}-a_2\frac{\mathcal{R}_1(z)}{\mathcal{R}_2(z)}e^{2g(z)}\equiv 0.
	\end{align*}In this case, we are also led to $\mathcal{Q}_{k-1}(z+c)=-\mathcal{Q}_{k-1}(z)$, yielding a contradiction.
\end{proof}	
\begin{lem}\label{lem2.3}
	For a transcendental meromorphic function  $ f $, if $  \Phi_{f(z)}^{(k)}\equiv \Phi^{(k)}_{f(z+c)}$ , then $ f $ is a periodic function of period $ c $ or $ 2c $ or $ 3c $. 
\end{lem}
\begin{proof}[\bf Proof of Lemma \ref{lem2.3}]
	By assumption, we have $\Phi_{f(z)}^{(k)}\equiv \Psi^{(k)}_{f(z+c)}$, \textit{i.e.,} 
	\begin{align*}
		\bigg[b_3f^3(z+c)+b_1f(z+c)\bigg]^{(k)}\equiv \bigg[b_3f^3(z)+b_1f(z)\bigg]^{(k)}.
	\end{align*} 
	Integrating $k$ times, we obtain 
	\begin{align}
		\label{e22.4} b_3f^3(z+c)+b_1f(z+c)\equiv b_3f^{3}(z)+b_1f(z)+\mathcal{P}_{k-1}(z),
	\end{align} 
	where $\mathcal{P}_{k-1}$ is a polynomial in $z$ of degree at most $k-1$. Thus, from \eqref{e22.4}, we get 
	\begin{align}
		\label{e22.5} [f(z+c)-f(z)][b_3f^2(z+c)+b_3f(z+c)f(z)+b_3f^2(z)+b_1]=\mathcal{P}_{k-1}(z).
	\end{align} 
	Now, we consider the following cases.
	\vspace{0.2cm}
	
	\noindent{\bf Case 1.} If $\mathcal{P}_{k-1}(z)\equiv 0$, it follows from \eqref{e22.5} that either $f(z+c)=f(z)$ (i.e., $f$ is periodic with period $c$), or 
	\begin{align}\label{e22.55} 
		b_3f^2(z+c)+b_3f(z+c)f(z)+b_3f^2(z)+b_1=0.
	\end{align} 
	By replacing $z$ with $z+c$, we must also have 
	\begin{align}\label{e22.66} 
		b_3f^2(z+2c)+b_3f(z+2c)f(z+c)+b_3f^2(z+c)+b_1=0.
	\end{align} 
	Subtracting \eqref{e22.55} from \eqref{e22.66} yields 
	\begin{align*}
		[f(z+2c)-f(z)][f(z+2c)+f(z+c)+f(z)]=0,
	\end{align*} 
	which implies that $f(z+2c)=f(z)$ or $f(z+2c)+f(z+c)+f(z)=0$. Standard periodicity arguments show that $f$ is either $2c$-periodic or $3c$-periodic.
	\vspace{0.2cm}
	
	\noindent{\bf Case 2.} If $\mathcal{P}_{k-1}(z) \not\equiv 0$, since $f(z)$ is a transcendental meromorphic function, it follows from \eqref{e22.5} that 
	\begin{align*}
		f(z+c)-f(z)&=\mathcal{R}_1(z)e^{h(z)},\\ 
		b_3f^2(z+c)+b_3f(z+c)f(z)+b_3f^2(z)+b_1&=\mathcal{R}_2(z)e^{-h(z)},
	\end{align*} 
	where $\mathcal{R}_1(z)$ and $\mathcal{R}_2(z)$ are two rational functions satisfying $\mathcal{R}_1(z)\mathcal{R}_2(z)=\mathcal{P}_{k-1}(z)$, and $h(z)$ is an entire function. An elementary calculation yields 
	\begin{align}\label{e222.5} 
		f(z+c)=\frac{3b_3\mathcal{R}_1(z)e^{h(z)}\pm \mathcal{D}(z)}{6b_3},\quad f(z)=\frac{-3b_3\mathcal{R}_1(z)e^{h(z)}\pm \mathcal{D}(z)}{6b_3},
	\end{align} 
	where $\mathcal{D}(z)$ satisfies 
	\begin{align*}
		\mathcal{D}^2(z)=9b_3^2\mathcal{R}^2_1(z)e^{2h(z)}-12b_3\left(b_3\mathcal{R}^2_1(z)e^{2h(z)}+b_1-\mathcal{R}_2(z)e^{-h(z)}\right).
	\end{align*}
	On the other hand, substituting $z+c$ into the expression for $f(z)$, we have 
	\begin{align}
		\label{e22.6} f(z+c)=\frac{-3b_3\mathcal{R}_1(z+c)e^{h(z+c)}\pm \mathcal{D}(z+c)}{6b_3}.
	\end{align} 
	Equating \eqref{e222.5} and \eqref{e22.6}, we see that 
	\begin{align}\label{e22.7} 
		\frac{-3b_3\mathcal{R}_1(z+c)e^{h(z+c)}\pm \mathcal{D}(z+c)}{6b_3}=\frac{3b_3\mathcal{R}_1(z)e^{h(z)}\pm \mathcal{D}(z)}{6b_3}.
	\end{align} 
	Since $\mathcal{R}_i(z)$ $(i=1, 2)$ are non-constant rational functions, they cannot be periodic. Consequently, $\mathcal{D}(z)$ is not periodic, which implies that identity \eqref{e22.7} cannot hold, yielding a contradiction. 
\end{proof}	
\section{\bf Proof of main results}
\subsection{\bf Proof of Theorem \ref{th1.1}}
Let $ f $ be a transcendental meromorphic function and $ \eta $ be non-zero Picard exceptional value. Then we can write $ f $ as \begin{align}\label{e3.1} f(z)=\eta+\frac{e^{\mathcal{Q}(z)}}{g(z)},
\end{align} where $ \mathcal{Q}(z) $ is a entire function such that $ T(r,\mathcal{Q})=S(r,f) $ and $ g $ is a non-constant non-exponential entire function. By our assumption, one can see that the function $ g $ must have zeros. Clearly, we see that \begin{align*}
	f^{(k)}(z)=\frac{\mathcal{H}_k\left(z,\mathcal{Q}(z),g(z)\right)}{g^{k+1}(z)},
\end{align*} where 
\begin{align*}
	\mathcal{H}_k\left(z,\mathcal{Q}(z),g(z)\right):= \mathcal{H}_k\left(z,\mathcal{Q}(z),\mathcal{Q}^{\prime}(z),\ldots,\mathcal{Q}^{(k)}(z),g(z),g^{\prime}(z),\ldots,g^{(k)}(z)\right)
\end{align*} is a differential polynomial in $ \mathcal{Q}(z) $ and $ g(z). $\vspace{1.2mm}

 Since $ ff^{(k)} $ is a periodic meromorphic function with period $ c $, then we have 
\begin{align}
	\label{e3.2} f(z)f^{(k)}(z)=f(z+c)f^{(k)}(z+c).
\end{align} Thus, we have 
\begin{align}
	\label{e3.3} & \frac{\mathcal{H}_k\left(z,\mathcal{Q}(z),g(z)\right)}{g^{k+2}(z)}e^{2\mathcal{Q}(z)}+\eta  \frac{\mathcal{H}_k\left(z,\mathcal{Q}(z),g(z)\right)}{g^{k+1}(z)}e^{\mathcal{Q}(z)}\\&= \frac{\mathcal{H}_k\left(z,\mathcal{Q}(z+c),g(z+c)\right)}{g^{k+2}(z+c)}e^{2\mathcal{Q}(z+c)}+\eta  \frac{\mathcal{H}_k\left(z,\mathcal{Q}(z+c),g(z+c)\right)}{g^{k+1}(z+c)}e^{\mathcal{Q}(z+c)}.\nonumber
\end{align} 
Re-writing (\ref{e3.3}), we see that 
\begin{align}
	\label{e3.4} f_1(z)+f_2(z)+f_3(z)=1,
\end{align} where 
\[
\begin{cases}
	f_1(z)\displaystyle=-\frac{1}{\eta}\frac{1}{g^{m}(z+c)}e^{\mathcal{Q}(z+c)}\\
	f_2(z)=\displaystyle\frac{1}{\eta}\frac{\mathcal{H}_k\left(z,\mathcal{Q}(z),g(z)\right)}{\mathcal{H}_k\left(z,\mathcal{Q}(z+c),g(z+c)\right)}\frac{g^{k+1}(z+c)}{g^{k+2}(z)}e^{2\mathcal{Q}(z)-\mathcal{Q}(z+c)}\\
	f_3(z) =\displaystyle
	\frac{\mathcal{H}_k\left(z,\mathcal{Q}(z),g(z)\right)}{\mathcal{H}_k\left(z,\mathcal{Q}(z+c),g(z+c)\right)}\frac{g^{k+1}(z+c)}{g^{k+1}(z)}e^{\mathcal{Q}(z)-\mathcal{Q}(z+c)}.
\end{cases}
\]
Evidently, the function $f_1$ is non-constant. Hence, Lemma \ref{lem2.1} yields the following two possibilities:\\

\noindent{\bf Possibility 1:-} 
\begin{align*}
	\displaystyle \frac{1}{\eta}\frac{\mathcal{H}_k\left(z,\mathcal{Q}(z),g(z)\right)}{\mathcal{H}_k\left(z,\mathcal{Q}(z+c),g(z+c)\right)}\frac{g^{k+1}(z+c)}{g^{k+2}(z)}e^{2\mathcal{Q}(z)-\mathcal{Q}(z+c)}\equiv 1
\end{align*} and 
\begin{align*}
	\frac{\mathcal{H}_k\left(z,\mathcal{Q}(z),g(z)\right)}{\mathcal{H}_k\left(z,\mathcal{Q}(z+c),g(z+c)\right)}\frac{g^{k+1}(z+c)}{g^{k+1}(z)}e^{\mathcal{Q}(z)-\mathcal{Q}(z+c)}-\frac{1}{\eta}\frac{1}{g(z+c)}e^{\mathcal{Q}(z+c)}=0. 
\end{align*} Combining these two, we obtain 
\begin{align}\label{e3.5} e^{\mathcal{Q}(z+c)+\mathcal{Q}(z)}=\eta^2 g(z)g(z+c).
\end{align} Since the function $g$ must have zeros by our assumption, we get a contradiction from \eqref{e3.5}.\\

\noindent{\bf Possibility 2:-} 
\begin{align*}
	\frac{\mathcal{H}_k\left(z,\mathcal{Q}(z),g(z)\right)}{\mathcal{H}_k\left(z,\mathcal{Q}(z+c),g(z+c)\right)}\frac{g^{k+1}(z+c)}{g^{k+1}(z)}e^{\mathcal{Q}(z)-\mathcal{Q}(z+c)}\equiv 1
\end{align*} and 
\begin{align*}
	\frac{1}{\eta}\frac{\mathcal{H}_k\left(z,\mathcal{Q}(z),g(z)\right)}{\mathcal{H}_k\left(z,\mathcal{Q}(z+c),g(z+c)\right)}\frac{g^{k+1}(z+c)}{g^{k+2}(z)}e^{2\mathcal{Q}(z)-\mathcal{Q}(z+c)}-\frac{1}{\eta}\frac{1}{g(z+c)}e^{\mathcal{Q}(z+c)}=0.
\end{align*} We therefore conclude that 
\begin{align}\label{e3.6} \frac{e^{\mathcal{Q}(z)}}{g(z)}=\frac{e^{\mathcal{Q}(z+c)}}{g(z+c)}.
\end{align} Finally, applying \eqref{e3.6} yields
\begin{align*}
f(z+c) = \eta + \frac{e^{\mathcal{Q}(z+c)}}{g(z+c)} = \eta + \frac{e^{\mathcal{Q}(z)}}{g(z)} = f(z).
\end{align*} 
Consequently, $f$ is a periodic meromorphic function with period $c$.
\subsection{\bf Proof of Theorem \ref{th1.2}.}
By hypothesis, $\left(a_2f^2(z)+a_1f(z)\right)^{(k)}$ is periodic. Thus, we have the identity 
\begin{align*}
	\Psi^{(k)}_{f}(z) = \Psi^{(k)}_{f}(z+c).
\end{align*} 
Applying Lemma~\ref{lem2.2}, it follows that the meromorphic function $f$ is periodic with period $c$ or $2c$, which completes the proof.
\subsection{\bf Proof of Theorem \ref{th1.3}.}
Since $\left(b_3f^3(z)+b_1f(z)\right)^{(k)}$ is a periodic function, we must have 
\begin{align*}
	\Phi^{(k)}_{f}(z) = \Phi^{(k)}_{f}(z+c).
\end{align*} 
By applying Lemma~\ref{lem2.3}, we obtain the desired conclusion.
\subsection{\bf Proof of Theorem \ref{th1.4}.} By our assumption, $\left(r_2\Psi^n_{f}(z)+r_1\Psi^{n-m}_{f}(z)\right)^{(k)}$ is periodic. Thus, we have 
\begin{align}\label{e3.7} 
	r_2\Psi^n_{f}(z)+r_1\Psi^{n-m}_{f}(z)=r_2\Psi^n_{f}(z+c)+r_1\Psi^{n-m}_{f}(z+c)+\mathcal{P}_1(z),
\end{align} 
where $\mathcal{P}_1(z)$ is a polynomial of degree at most $k-1$.
\vspace{0.2cm}

\noindent{\bf Case 1.} If $\mathcal{P}_1(z)\not\equiv 0$, then by Theorem E, there exists a transcendental meromorphic function $f$ satisfying equation \eqref{e3.7}. 
\vspace{0.2cm}

\noindent{\bf Case 2.} If $\mathcal{P}_1(z)\equiv 0$, then applying Theorem D yields $\Psi_{f}(z)=h\Psi_{f}(z+c)$, where $h^n=1$. Clearly, $\Psi_{f}(z)$ is periodic with period $nc$, i.e., $\Psi_{f}(z+nc)=\Psi_{f}(z)$ for all $z \in \mathbb{C}$. It follows that 
\begin{align*}
	a_2f^n(z+nc)+a_1f(z+nc)=a_2f^n(z)+a_1f(z).
\end{align*} 
By adopting the previous arguments in the proof of Lemma~\ref{lem2.2}, it follows that $f$ is a periodic function with period $nc$ or $2nc$. 
\subsection{Proof of Theorem \ref{th1.5}} 
Proceeding in exactly the same manner as in the proof of Theorem~\ref{th1.4} and utilizing Lemma~\ref{lem2.3}, we obtain the conclusion of Theorem~\ref{th1.5}.
\vspace{5mm}

\noindent\textbf{Ethical Approval.} Not applicable for this work.\\
\noindent\textbf{Competing interests.} The authors have not competing interest.\\
\noindent\noindent\textbf{Funding.} No funding for this research work.\\  
\noindent\textbf{Conflict of interest.} The authors declare that there is no conflict  of interest regarding the publication of this paper. \\
\noindent\textbf{Data availability statement.}  Data sharing not applicable to this article as no datasets were generated or analysed during the current study.

\end{document}